\documentclass[11pt]{amsart}

\usepackage{times}
\usepackage[utf8]{inputenc}
\usepackage{graphicx}
\usepackage{amsmath}
\usepackage{amsthm}
\usepackage{amssymb}
\usepackage{amsfonts}

\usepackage{xypic}

\usepackage[hidelinks]{hyperref}
\usepackage[capitalise,nameinlink]{cleveref}

{\theoremstyle{plain}
\newtheorem{theorem}{Theorem}[section]
\newtheorem{lemma}[theorem]{Lemma}
\newtheorem{proposition}[theorem]{Proposition}
\newtheorem{corollary}[theorem]{Corollary}

}

{\theoremstyle{definition}
\newtheorem{definition}[theorem]{Definition}
}

\newcommand{\defemph}[1]{\emph{\textbf{#1}}}

\newcommand{\defined}{{\downarrow}}

\newcommand{\pair}[1]{\langle #1 \rangle}

\newcommand{\NN}{\mathbb{N}}
\newcommand{\NNx}{{\NN_\infty}}

\newcommand{\wCPO}{$\omega$-cpo}
\newcommand{\wCPOs}{$\omega$-cpos}
\newcommand{\Sup}[1]{{\textstyle \bigvee_{#1}}}
\newcommand{\ScottOpen}[1][]{\mathcal{S}_{#1}}
\newcommand{\Compact}[1]{\mathcal{K}(#1)}
\newcommand{\wayb}{\ll}

\newcommand{\parto}{\rightharpoonup}
\newcommand{\multito}{\rightrightarrows}
\newcommand{\inhab}[1]{\mathcal{P}_{\!{*}}(#1)}
\newcommand{\supp}[1]{\|#1\|}

\newcommand{\inv}[1]{#1^{*}}

\newcommand{\dom}[1]{\mathrm{dom}(#1)}

\newcommand{\set}[1]{\{#1\}}
\newcommand{\such}{\mid}

\newcommand{\one}{\mathsf{1}}
\newcommand{\two}{\mathsf{2}}

\newcommand{\upper}{{\uparrow}}

\newcommand{\cf}{\varphi}

\newcommand{\pcm}{\NN \to \lift{\NN}} 

\newcommand{\liff}{\Leftrightarrow}
\newcommand{\lthen}{\Rightarrow}

\newcommand{\lift}[1]{{#1}_{\!\bot}}

\newcommand{\Hom}[1]{\mathsf{Hom}(#1)}

\newcommand{\nbhmap}[1][]{\Phi_{#1}}
\newcommand{\nbh}[2][]{\nbhmap[#1](#2)}

\newcommand{\all}[1]{\forall #1.\,}
\newcommand{\some}[1]{\exists #1.\,}

\newcommand{\apart}{\mathbin{\#}}

\newcommand{\topol}[1]{\langle #1 \rangle}

\title[Spreen spaces and the KLST theorem]{Spreen spaces and the synthetic\\ Kreisel-Lacombe-Shoenfield-Tseitin theorem}
\author{Andrej Bauer}
\email{Andrej.Bauer@andrej.com}
\urladdr{https://www.andrej.com/}
\address{Faculty of Mathematics and Physics, University of Ljubljana, Slovenia}
\address{Institute of Mathematics, Physics and Mechanics, Ljubljana, Slovenia}

\thanks{This material is based upon work supported by the Air Force Office of Scientific Research under award number FA9550-21-1-0024.}

\subjclass[2020]{Primary 03F60; Secondary 03F55}

\begin{document}

\begin{abstract}
  I take a constructive look at Dieter Spreen's treatment of effective topological spaces and the Kreisel-Lacombe-Shoenfield-Tseitin (KLST) continuity theorem.
  Transferring Spreen's ideas from classical computability theory and numbered sets to a constructive setting
  leads to a theory of topological spaces, in fact two of them: a locale-theoretic one embodied by the notion of $\sigma$-frames, and a pointwise one that follows more closely traditional topology.
  Spreen's notion of effective limit passing turns out to be closely related to sobriety,
  while his witnesses for non-inclusion give rise to a novel separation property --
  any point separated from an overt subset by a semidecidable subset is already separated from it by an open one.
  I name spaces with this property \emph{Spreen spaces}, and show that they give rise to a purely constructive continuity theorem: every map from an overt Spreen space to a pointwise regular space is pointwise continuous. The theorem is easily proved, but finding non-trivial examples of Spreen spaces is harder.
  I show that they are plentiful in synthetic computability theory.
\end{abstract}

\maketitle

\section{Introduction}
\label{sec:introduction}

Dieter Spreen gave a comprehensive treatment~\cite{Spreen98} of topological spaces in the context of numbered sets and computability theory. He defines an effective topological space to be a topological space with given numberings of its points and basic opens, equipped with a suitable additional computability structure. The definition is quite general, as it encompasses recursive metric spaces, effective domains, and other kinds of computable spaces. Among many results we find a novel generalization of the fundamental theorem of Kreisel, Lacombe, Shoenfield and Tseitin (KLST) about effective pointwise continuity of effective functions.

In this paper I transfer Spreen's work to constructive mathematics and to the setting of synthetic computability~\cite{Bauer06,BauerLesnik12,bauer17}.
The result is a satisfying account of countably based topological spaces (\cref{sec:topological-spaces}). 
There are in fact two variants, a localic one in the style of synthetic topology~\cite{escardo04:_synth,Lesnik10,BauerLesnik12}, as embodied by the notion of $\sigma$-frames (\cref{sec:sigma-frames-malcev}), and a pointwise one that follows more closely traditional topology (\cref{sec:pointwise-notions}). These are shown to be inequivalent in the effective topos  (\cref{prop:Spreen_equiv_RE}).

Effective topological notions, such as effective separability, effective regularity, effective continuity, etc., readily translate to standard topological concepts. It is less obvious that Spreen's effective limit passing is closely related to sobriety (\cref{prop:cb-sober-limit-passing}), and that his witnesses for non-inclusion turn out to be a new separation property involving overt subsets. We call spaces satisfying it \emph{Spreen spaces} (\cref{def:spreen-space}).

In constructive mathematics Spreen's formulation of the KLST theorem becomes a continuity principle stating that all maps from an overt Spreen space to a regular spaces are pointwise continuous (\cref{thm:synthetic-klst}). The proof is very simple, even trivial, but to show that all countably based sober spaces are Spreen spaces (\cref{thm:cb-sober-spreen}) we have to work harder and employ the full power of the axioms of synthetic computability (\cref{sec:synth-comp}). We must even adopt a new synthetic axiom, the Stable Subspace Axiom, to get the desired results.

Proofs and constructions in the theory of numbered sets are intensional in nature, i.e., one refers to realizers of elements and computations on them, thereby performing steps that need not respect equality of mathematical objects. In the synthetic setting this approach is not available, as one works directly with mathematical objects. To cope with this phenomenon, we employ both partial and multivalued maps (\cref{sec:partial-multi-valued}). We use the synthetic Recursion theorem (\cref{thm:recursion-theorem}), which provides fixed points of multivalued maps, to derive a sequential continuity principle for multivalued open sets (\cref{thm:multi-wso}). The principle allows us to recast the so called ``waiting arguments'' of computability theory to constructive statements about sequential limits.

Unless stated otherwise, we work constructively~\cite{bishop67} in higher-order intuitionistic logic with Dependent Choice. We strive to minimize the uses of Dependent Choice by replacing it with Countable Choice whenever possible, and never silently apply any non-logical axioms.
Even though we habitually speak of sets, everything we do can be interpreted in the internal language of a topos, so long as universal quantifications over all sets or objects of a certain kind are understood schematically as statements about all objects of a topos.

\section{Semidecidable truth values and subsets}
\label{sec:semid-truth-valu}

Under the traditional view of topology the open subsets of a topological space may be arbitrarily complex. In an extreme case, the topology comprises the whole powerset of the underlying set, which is much too liberal from the point of view of computability theory, where it is natural to require that open sets be semidecidable. So we first need a notion of semidecidable sets which works well constructively.

The \defemph{Rosolini dominance $\Sigma$}~\cite{rosolini86} is the set of those truth values that are equivalent to existential quantification of a decidable predicate on $\NN$,
\begin{equation*}
  \Sigma = \set{
    p \in \Omega \such
    \some{f \in 2^\NN} (p \liff \some{n \in \NN} f(n) = 1)
  },
\end{equation*}
where $\Omega$ is the set of all truth values (the subobject classifier).
We call the elements of $\Sigma$ the \defemph{semidecidable truth values}. The name is justified because $\Sigma$ contains $\bot$ and $\top$, and is closed under finite meets and countable joins, thanks to Countable Choice.
For any set $X$ we call the elements of $\Sigma^X$ the \defemph{semidecidable subsets} of~$X$. (In general we write $Y^X$ or $X \to Y$ for the set of all functions from $X$ to $Y$.)

We remark that in the effective topos $\Sigma$ is the numbered set with two elements~$\bot$ and~$\top$, realized by the (codes of) the non-terminating and terminating computations, respectively.  If $X$ is a numbered set then the exponential $\Sigma^X$ is the numbered set of completely computably enumerable subsets of~$X$, with and acceptable numbering. In classical mathematics $\Sigma = \Omega$, because all truth values are decidable and hence also semidecidable, and $\Sigma^X$ is just the powerset of~$X$.

A notion related to semidecidability is overtness. Say that a subset $T \subseteq X$ is \defemph{overt} when existential quantification over~$T$ preserves semidecidability, i.e., for any $\phi : X  \to \Sigma$ the truth value $\some{x \in T} \phi(x)$ is semidecidable. We say that a space $X$ is overt when it is overt as a subset of itself. The natural numbers are overt, and under further assumptions also separable sober spaces, see \cref{prop:separable-overt}. Synthetic topology has much to say about overtness, which is dual to compactness, but we shall limit ourselves just to two observations: the inverse image of an overt subset is overt, and a semidecidable subset of an overt set is overt.

A set $X$ is \defemph{countable} when there is a surjection $e : \NN \to \one + X$, called an \defemph{enumeration} of~$X$, where $\one = \set{\star}$ is a singleton and $+$ stands for the disjoint sum. Thus an enumeration is allowed to output the special element $\star$ when it wants to ``skip''. An enumeration which always skips enumerates the empty set. An inhabited set may be enumerated without skipping.

The semidecidable subsets of $\NN$ are precisely the countable ones. Indeed, if $S \subseteq \NN$ is enumerated by $e$ then $n \in S$ is semidecidable, as it is equivalent to $\some{m \in \NN} n = e(m)$. Conversely, if $S \subseteq \NN$ is semidecidable, then there exists by countable choice a map $p : \NN \to \two^\NN$ such that $m \in S$ is equivalent to $\some{n \in \NN} p(m)(n) = 1$, hence an enumeration of $S$ is given by
\begin{equation*}
  e (\pair{m, n}) =
  \begin{cases}
    m & \text{if $p(m)(n) = 1$,}\\
    \star & \text{if $p(m)(n) = 0$.}
  \end{cases}
\end{equation*}
where $\pair{{-}, {-}}$ is a pairing function witnessing that $\NN \times \NN$ is isomorphic to $\NN$.

\section{Topological spaces}
\label{sec:topological-spaces}

A fundamental insight of synthetic topology is that the semidecidable subsets $\Sigma^X$ behave like a topology on~$X$. Because $\Sigma^X$ is not something we impose, but is rather present already, we call $\Sigma^X$ the \defemph{intrinsic topology} of~$X$. In Spreen's theory of effective spaces the intrinsic topology of a numbered set is known as the \emph{Ershov topology}, while a \emph{Malcev topology} is one that is coarser than the Ershov topology.

\subsection{Malcev topologies and $\sigma$-frames}
\label{sec:sigma-frames-malcev}

We should be careful about forming unions of open sets, because arbitrary unions need not preserve semidecidability.
They are preserved by countable unions, though, which prompts the following definition.

\begin{definition}
  A \defemph{$\sigma$-frame} is a partially ordered set with finite meets and countable joins, with the former distributing over the latter. A \defemph{Malcev topology} on a set $X$ is a $\sigma$-subframe of the intrinsic topology~$\Sigma^X$.
\end{definition}

Henceforth we call Malcev topologies just ``topologies'', as these are the only ones we are interested in. We write $(X, \mathcal{T})$ for a set $X$ with a Malcev topology $\mathcal{T}$. Next, we need a notion of continuity.

\begin{definition}
  A \defemph{$\sigma$-homomorphism} is a map between $\sigma$-frames which preserves finite meets and countable joins.
  A map $f : X \to Y$ between spaces $(X, \mathcal{T}_X)$ and $(Y, \mathcal{T}_Y)$ is \defemph{continuous} if
  its inverse image map $\inv{f}$ maps $\mathcal{T}_Y$-open sets to $\mathcal{T}_X$-open sets, in which case
  it is a $\sigma$-homomorphism $\inv{f} : \mathcal{T}_Y \to \mathcal{T}_X$.
\end{definition}

We write $\Hom{L,M}$ for the set of $\sigma$-homomorphisms between $\sigma$-frames~$L$ and~$M$.
Note that the meets and joins of a Malcev topology are computed as intersections and unions, respectively, because they are inherited from the intrinsic topology. The prototypical example of a topology is the metric topology in which we take as the open sets the countable unions of open balls. The open balls are semidecidable because strict inequality~$<$ on $\mathbb{R}$ is semidecidable.

If a metric space is separable, its topology is generated by a countable collection of open balls. In general, we shall be interested in those topologies that are generated by a countable family of opens.

\begin{definition}
  A \defemph{countable base} for a $\sigma$-frame $L$ is a family $(b_n)_{n \in \NN}$ of its elements, which we call \defemph{basic}, such that, for every $x \in L$, there is a countable set $I \subseteq \NN$ for which $x = \bigvee_{i \in I} b_i$.
\end{definition}

\noindent
When is a family of semidecidable sets the base of a topology? The answer is analogous to the familiar one from classical topology.

\begin{proposition}
  A family $(B_n)_{n \in \NN}$ of semidecidable subsets of $X$ is the base of a topology on $X$ if, and only if, the intersection of any two basic sets $B_i \cap B_j$ is a countable union of basic sets. In this case the unique topology $\topol{B}$ which has $B$ as its base is the collection of all countable unions of the basic sets.
\end{proposition}

\begin{proof}
  Clearly, if $(B_n)_{n \in \NN}$ is the base of a topology, then the intersection of any two basic sets, itself being an element of the topology, must be a countable union of basic sets. Conversely, if $(B_n)_{n \in \NN}$ has the required property, then the collection of all countable unions of its elements forms a $\sigma$-frame because it is closed under finite intersections by the required property. By design $(B_n)_{n \in \NN}$ is its base.
\end{proof}

A \defemph{countably based space $(X, \topol{B})$} is a set $X$ with a topology that is generated by a chosen countable base $B : \NN \to \Sigma^X$.

\subsection{Pointwise topology and continuity}
\label{sec:pointwise-notions}

We may formulate topological notions by emphasizing the points. Given a family $(B_i)_{i \in I}$ of semidecidable subsets of~$X$, say that $U \subseteq X$ is \defemph{pointwise open} with respect to $(B_i)_{i \in I}$ if for every $x \in U$ there is $i \in I$ such that $x \in B_i \subseteq U$. In other words, whereas open sets are overt unions of basic opens, the pointwise open sets are arbitrary unions of opens.

We say that the family $(B_i)_{i \in I}$ is a \defemph{pointwise base} if $B_i \cap B_j$ is pointwise open with respect to the family, for all $i, j \in I$. The topology generated by such a base consists of all unions of basic opens, and is closed under finite intersections and all unions.
A map is \defemph{pointwise continuous} when its inverse image takes pointwise open set to pointwise open sets.






An obvious question to ask is whether passing to the pointwise notions makes a difference,
since it does not in classical mathematics.

\begin{proposition}
  \label{prop:Spreen_equiv_RE}%
  In the effective topos there is a countable pointwise base which is not a countable base.
\end{proposition}

\begin{proof}
  In this proof we work in classical mathematics, and in particular classical computability theory. Recall from computability theory that there exists an infinite and coinfinite computably enumerable (c.e.) set $M$ such that whenever $M \subseteq K$ and $K$ is~c.e., then $K\setminus M$ or $\NN \setminus K$ is finite. Such a set is called a \emph{maximal c.e.~set}.\footnote{I thank Douglas Cenzer for identifying maximal c.e.~sets as the essential ingredient of this proof.} A proof of their existence may be found in~\cite[Chap.~X, Sect.~3]{Soare:87}.

  We are going to construct the relevant objects of the effective topos as numbered sets. Let $I$ be a superset of $M$ that is coinfinite and not~c.e. There are coinfinite sets $U$ and $V$ that are not~c.e.\ such that $U \cup V$ is coinfinite, $I = U \cap V$, $U \setminus I$ is infinite, and $V \setminus I$ is infinite. Let $X$ be a superset of $U \cup V$ that is not~c.e., is coinfinite, and $X \setminus (U \cup V)$ is infinite.

  As our numbered set of points we take $X$ with the numbering $\nu_X : \NN \parto X$ defined by
 \begin{align*}
    \nu_X (4 n) = n &\iff n \in X \setminus (U \cup V),\\
    \nu_X (4 n + 1) = n &\iff n \in U,\\
    \nu_X (4 n + 2) = n &\iff n \in V,\\
    \nu_X (4 n + 3) = n &\iff n \in U \cap V.
  \end{align*}
  If $0 \leq j \leq 3$ and $n \not\in X$ then $\nu_X(4 n + j)$ is undefined. Next, we need a pointwise base $B : \NN \to \Sigma^X$ on $X$, which we define as:
  \begin{align*}
    B_0 &= U,
    &
    B_1 &= V,
    &
    B_{n+2} &= \set{n} \cap X.
  \end{align*}
  Let us verify that $B$ is a pointwise base on $X$. The relation $\nu_X(i) \in B_n$ is completely~c.e.\ because, for all $n, k \in \NN$ and $i, j \in \set{0,1}$ such that $4 k + 2 j + i \in \dom{\nu_X}$,
  \begin{equation*}
    \nu_X({4 k + 2 j + i}) \in B_n \iff
    \begin{aligned}[t]
      &(n = 0 \Rightarrow i = 1) \ \land \\
      &(n = 1 \Rightarrow j = 1) \ \land \\
      &(n \geq 2 \Rightarrow k = n - 2).
    \end{aligned}
  \end{equation*}
  The right-hand side is even decidable. So $B_n$ is indeed a semidecidable subset of $(X, \nu_X)$, uniformly in~$n$. Let $s : \NN^3 \parto \NN$ be the function defined for $i, j \in \set{0, 1}$, $k, m, n \in \NN$ by
  \begin{equation*}
    s (4 k + 2 j + i, m, n) = k + 2.
  \end{equation*}
  Suppose $\nu_X(4 k + 2 j + i)$ is defined and $\nu_X(4 k + 2 j + i) \in B_m \cap B_n$. Then clearly we have
  \begin{equation*}
    \nu_X(4 k + 2 j + 1) = k \in \set{k} = B_{s(4 j + 2 j + i, m, n)}
    \subseteq B_m \cap B_n.
  \end{equation*}
  This shows that $B$ is a pointwise base.
  
  To see that $B$ is not a base, assume to the contrary that it is. Then for some total recursive function $r : \NN \times \NN \to \NN$ we would have
  \begin{equation*}
    I = U \cap V = B_0 \cap B_1
    = \bigcup \set{B_k \such k \in W_{r(0,1)}},
  \end{equation*}
  where $W$ is a standard numbering of c.e.~sets. Since $I$ contains neither $U$ nor $V$, it follows that $0 \not\in W_{r(0,1)}$ and $1 \not\in W_{r(0,1)}$. The set
  \begin{equation*}
    C = \set{k - 2 \such k \in W_{r(0,1)}}
  \end{equation*}
  is a c.e.~set and $M \subseteq I \subseteq C$. Since $I \setminus M$ is infinite, $C \setminus M$ is infinite. It follows from maximality of~$M$ that $\NN \setminus C$ is finite, but this is only possible if $C \cap (X \setminus I) \neq \emptyset$. Pick some $k \in C \cap (X \setminus I)$. Then on one hand $k \in X \setminus I$, and on the other $k \in B_{k+2} \in I$ because $k+2 \in W_{r(0,1)}$. This is a contradiction.
\end{proof}

\subsection{The neighborhood filter and sober spaces}
\label{sec:neighb-filt-sober}

We may ask under what circumstances the points of a space $(X, \mathcal{T})$ can be recovered from its topology.
Define the \defemph{neighborhood filter} $\hat{\Box} : X \to \Hom{\mathcal{T}, \Sigma}$ by $\hat{x}(U) = (x \in U)$ for $x \in X$ and $U \in \mathcal{T}$. The $\sigma$-homomorphism $\hat{x}$ contains as much information about~$x$ as $\mathcal{T}$ can provide. Thus, if the neighborhood filter can be inverted, the points of $X$ can be recovered from $\mathcal{T}$.

\begin{definition}
  A space is \defemph{sober} if its neighborhood filter is a bijection.
\end{definition}

For a countably based space $(X, \topol{B})$ we can also define the \defemph{basic neighborhood filter} $\nbhmap[X] : X \to \Sigma^\NN$ as the map
\begin{equation*}
  \nbh[X]{x} = \hat{x} \circ B = \set{i \in \NN \such x \in B_i},
\end{equation*}
which assigns to a point the set of indices of its basic neighborhoods. When no confusion could arise, we write $\nbhmap$ instead of $\nbhmap[X]$.
The neighborhood filter may be recovered from the basic neighborhood filter. Indeed, for any $x \in X$ and $U \in \topol{B}$ we have
\begin{equation*}
  \hat{x}(U) = (\some{n \in \NN} B_n \subseteq U \land \nbh{x}).
\end{equation*}

We give two examples of sober spaces, one arising from metrics and the other from partial orders.

\begin{proposition}
  \label{prop:csm-sober}%
  A complete separable metric space is sober with respect to the metric topology.
\end{proposition}

\begin{proof}
  Suppose $(M, d)$ is a complete metric space with a dense sequence $(a_i)_{i \in \NN}$. The metric topology has a countable base of open balls of the form $B_{i,j} = B(a_i, 2^{-j})$, where $i, j \in \NN$.
  Consider an arbitrary $\sigma$-homomorphism $p : \topol{B} \to \Sigma$. Because every ball in~$M$ can be written as a countable union of basic balls, Dependent Choice yields a sequence $(i_k)_{k \in \NN}$ such that $B_{i_{k+1}, k+1} \subseteq B_{i_k, k}$ and $p (B_{i_k,k}) = \top$, for all $k \in \NN$.
  Also note that the sequence $(a_{i_k})_k$ is Cauchy, so it converges to a point $x = \lim_k a_{i_k}$ which is contained in all $B_{i_k, k}$'s.

  We claim that $p = \hat{x}$. To show this, consider any open ball $B(y, r)$. If $x \in B(y, r)$ then there is~$k$ such that $x \in B_{i_k, k} \subseteq B(y, r)$, hence $p(B(y, r)) \geq p(B_{i_k,k}) = \top$.
  For the other direction, suppose $p(B(y, r)) = \top$. There is $0 < q < r$ such that
  $p(B(y,q)) = \top$. Then, for any $k \in \NN$,
  \begin{equation*}
    p(B(y,q) \cap B_{i_k, k}) =
    p(B(y,q)) \land p(B_{i_k, k}) = \top,
  \end{equation*}
  whence $B(y,q)$ and $B_{i_k, k}$ cannot be disjoint, from which $d(y, a_{i_k}) \leq q + 2^{-k}$ follows. For a large enough~$k$ we obtain $d(y, a_{i_k}) \leq q + 2^{-k} < r$, and so $x \in B(y, r)$, which implies $\hat{x}(B(y,r)) = \top$, as desired.
\end{proof}

Completeness cannot be dropped from the previous proposition, in contrast to classical mathematics where all Hausdorff spaces are sober.

\begin{proposition}
  \label{prop:lpo-if-all-separable-metric-sober}
  If every separable metric space is sober then the Limited Principle of Omniscience (LPO) holds.
\end{proposition}

\begin{proof}
  The limited principle of omniscience states that every $\alpha \in \two^\NN$ is either equal to the zero sequence $o$, or is apart from it. Recall that the apartness relation $\alpha \apart \beta$ is defined as $\some{n \in \NN} \alpha(n) \neq \beta(n)$.

  Let $d$ be the complete ultrametric on the Cantor space $\two^\NN$ defined by
  \begin{equation}
    \label{eq:cantor-metric}%
    d(\alpha, \beta) = \lim\nolimits_n 2^{-\min \set{k \such k = n \lor \alpha_k \neq \beta_k}}.
  \end{equation}
  We write $\overline{\alpha}(n) = [\alpha(0), \ldots, \alpha(n-1)]$ for the initial segment of $\alpha$ of length~$n$. A countable basis for the metric topology on $\two^\NN$ consists of open balls, one for each finite sequence $a \in 2^{*}$ of length $|a|$,
  \begin{equation*}
    B_a = \set{\alpha \in \two^\NN \such \overline{\alpha}(|a|) = a}.
  \end{equation*}
  Define the metric space
  $M = \set{ \alpha \in \two^\NN \such \alpha = o \lor \alpha \apart o }$,
  with the metric induced by~$\two^\NN$. The open balls $B'_a = B_a \cap M$ form a countable base for the metric topology on $M$. Because $M$ is dense in $\two^\NN$ we may define a $\sigma$-homomorphism $r : \topol{B'} \to \topol{B}$ by
  \begin{equation*}
    \textstyle
    r \left( \bigvee_{i \in \NN} B'_{a_i} \right) =
    \bigvee_{i \in \NN} B_{a_i}.
  \end{equation*}
  In fact, $r$ is the left inverse of the $\sigma$-homomorphism $i^{*} : \topol{B} \to \topol{B'}$ where $i : M \to \two^\NN$ is the canonical inclusion.

  Consider any $\alpha \in \two^\NN$. By assumption $M$ is sober, therefore there exists $\beta \in M$ such that $\hat{\beta} = \hat{\alpha} \circ r$. But then $\hat{\alpha} = \hat{\alpha} \circ r \circ i^{*} = \hat{\beta} \circ i^{*}$, which means that $\alpha = \beta$, so $\alpha = o$ or $\alpha \apart o$.
\end{proof}

For our second example of sober spaces we venture into domain theory. Dieter Spreen considered general spaces, such as $A$-spaces, $f$-spaces, and other domain-like spaces. We shall be content with just one kind of
domains, namely the $\omega$-algebraic chain-complete partial orders. We review here just enough domain theory to establish sobriety of such spaces, and direct the interested readers to~\cite{bauer17} for a further study of domain theory in the context of synthetic computability.

A \defemph{partially ordered set}, or \defemph{poset}, $(P, {\leq})$ is a set~$P$ with a reflexive, transitive and asymmetric relation~$\leq$.
A \defemph{chain} in $(P, {\leq})$ is a a monotone sequence $c : \NN \to P$: for all $i \in \NN$, $c_i \leq c_{i+1}$.
A \defemph{chain-complete poset ({\wCPO})} is a poset $(P, {\leq})$ in which every chain $c : \NN \to P$ has a supremum $\Sup{n} c_n$.
%

Recall that $S \subseteq P$ is \defemph{directed} if it is inhabited and for all $x, y \in S$ there is $z \in S$ such that $x \leq z$ and $y \leq z$. The following lemma shows that in the presence of Countable Choice no generality is gained by generalizing chains to countable directed-complete sets.

\begin{lemma}
  \label{lemma:dcpo-from-wcpo}
  A countable directed subset of an {\wCPO} has a supremum.
\end{lemma}

\begin{proof}
  Suppose $S \subseteq P$ is a directed subset of an {\wCPO} $(P, \leq)$, enumerated by $e : \NN \to S$.
  By Countable Choice there is a map $s : \NN \times \NN \to \NN$ such that $e_m \leq e_{s(m,n)}$ and $e_n \leq e_{s(m,n)}$ for all $m, n \in \NN$. Define $c : \NN \to \NN$ by
  $c_0 = 0$ and
  $c_{n+1} = s(c_n, n)$.
  Then $e \circ c$ is a cofinal chain in~$S$, so its supremum is the supremum of~$S$.
\end{proof}

We say that a semidecidable subset $U \subseteq P$ is \defemph{Scott-open} when it is
\begin{enumerate}
\item upward closed: if $x \in U$ and $x \leq y$ then $y \in U$, and
\item inaccessible by suprema of chains: if $c : \NN \to P$ is a chain and $\Sup{n} c_n \in U$ then $c_m \in U$ for some $m \in \NN$.
\end{enumerate}
The Scott-opens of an {\wCPO} $(P, {\leq})$ form a $\sigma$-frame $\ScottOpen[P]$, as is easily checked.

In a poset $(P, {\leq})$, the \defemph{way-below relation} $x \wayb y$ is defined to mean: if $c : \NN \to P$ is a chain such that $y \leq \Sup{n} c_n$ then there is $m \in \NN$ such that $x \leq c_m$. An element $x \in P$ is \defemph{compact} when $x \wayb x$.
We let $\Compact{P}$ denote the set of all compact elements of~$P$. It is closed under binary joins.

An {\wCPO} $(P, {\leq})$ is \defemph{$\omega$-algebraic} when
\begin{enumerate}
\item $\Compact{P}$ is countable and the induced order decidable, and
\item for every $x \in P$, the set $\set{y \in \Compact{P} \such y \leq x}$ is countable and~$x$ is its supremum.
\end{enumerate}
When these conditions are met, every $x \in P$ is the supremum of a chain of compact elements, because $\set{y \in \Compact{P} \such y \leq x}$ contains a cofinal chain, as was shown in \cref{lemma:dcpo-from-wcpo}. Furthermore, for $x \in \Compact{P}$ and $y \in P$ the relation $x \leq y$ is semidecidable: there is a chain $c : \NN \to \Compact{P}$ whose supremum is~$y$, hence $x \leq y$ is equivalent to $\some{n \in \NN} x \leq c_n$, and $x \leq c_n$ is decidable.

The upshot of these definitions is that the Scott topology of an $\omega$-algebraic {\wCPO} $P$ is countably based, with the basic opens of the form $\upper x$ for $x \in \Compact{P}$. We just argued that they $\upper x$ is semidecidable, and it is Scott-open because~$x$ is compact. To see that we really have a base, we verify that an Scott-open $U \subseteq P$ is a countable union of basic opens
\begin{equation*}
  \textstyle
  U = \bigcup \set{\upper x \such x \in \Compact{P} \land x \in U}.
\end{equation*}
Clearly, this is a countable union because $\Compact{P}$ is cuntable and $x \in U$ semidecidable.
The inclusion from right to left holds because~$U$ is an upper set.
For the opposite inclusion, consider any $y \in U$. Because~$y$ is the supremum of compact elements below it, and $U$ is inaccessible by suprema of chains, there is $x \in \Compact{P}$ such that $x \leq y$ and $x \in U$, therefore $y \in \upper{x}$ and $y$ is an element of the right-hand side.

\begin{proposition}
  The Scott topology of an $\omega$-algebraic {\wCPO} is sober.
\end{proposition}

\begin{proof}
  Let $(P, {\leq})$ be an $\omega$-algebraic {\wCPO}, and consider an arbitrary $\sigma$-homomorphism $p : \ScottOpen[P] \to \Sigma$.
  The set
  \begin{equation*}
    S = \set{x \in \Compact{P} \such p(x) = \top}
  \end{equation*}
  is countable and directed, for if $p(x) = \top$ and $p(y) = \top$ for $x, y \in \Compact{P}$ then
  \begin{equation*}
    p(\upper (x \lor y)) =
    p(\upper x \cap \upper y) =
    p(\upper x) \land p(\upper y) =
    \top \land \top =
    \top.
  \end{equation*}
  We claim that $p = \hat{z}$ where $z = \bigvee S$.
  It suffices to check that $p$ and $\hat{z}$ agree on the basic opens.
  For any $x \in \Compact{P}$, if $\hat{z}(\upper x)$ then $x \leq z = \bigvee S$, hence $x \leq y$ for some $y \in S$,
  from which we get $\top = p(\upper y) \leq p(\upper x)$. Conversely, if $p(\upper x) = \top$ then $x \in S$, hence $x \leq \bigvee S = z$ and $\hat{z}(\upper x) = \top$.
\end{proof}

The poset $(\Sigma^\NN, {\subseteq})$ is an $\omega$-algebraic {\wCPO} whose compact elements are the finite subset of~$\NN$.
This is in fact a universal countably based $T_0$-space.

\begin{definition}
  A space is a \defemph{$T_0$-space} when its (basic) neighborhood filter is injective.
\end{definition}

\begin{proposition}
  \label{prop:sigma-N-universal}%
  Every countably-based $T_0$-space embeds in $\Sigma^\NN$ equipped with the Scott topology.
\end{proposition}

\begin{proof}
    The embedding of a countably based $T_0$-space $(X, \topol{B})$ into $\Sigma^\NN$ is the basic neighborhood filter $\nbhmap : X \to \Sigma^\NN$. It is injective because $X$ is a $T_0$-space, and is continuous because the inverse image of $\upper \set{n_1, \ldots, n_k}$ is $B_{n_1} \cap \cdots \cap B_{n_k}$. To see that $\inv{\nbhmap}$ is surjective, and therefore $\nbhmap$ a topological embedding, consider $U \in \topol{B}$ and write it as a union $U = \bigcup_{n \in I} B_n$ where $I \subseteq \NN$ is countable. Then
    \begin{equation*}
      \textstyle
      \nbh{x} \in \bigcup_{n \in I} \upper\set{n}
      \iff
      \some{n \in I} x \in B_n
      \iff
      x \in U,
    \end{equation*}
    therefore $\inv{\nbhmap}(\bigcup_{n \in I} \upper\set{n}) = U$.
\end{proof}

\subsection{An equational characterization of $\sigma$-frame homomorphism}
\label{sec:an-equat-char}

Given a $\sigma$-frame with a countable base, how can we tell whether a map of basic elements into another $\sigma$-frame is induced by a $\sigma$-homomorphism? The following answer shows that the criterion can be expressed equationally.

\begin{lemma}
  \label{lemma:when-induced-by-homomorphism}
  Let $L$ be a $\sigma$-frame with a countable base $(b_n)_{n \in \NN}$ and $\phi : L \to M$ a $\sigma$-homomorphism. Then the map $f : \NN \to M$ defined by $f(n) = \phi(b_n)$ satisfies the following conditions:
  \begin{enumerate}
  \item for all $T, T' \in \Sigma^\NN$, if $\bigvee_{k \in T} b_k = \bigvee_{m \in T'}
    b_m$ then $\bigvee_{k \in T} f(k) = \bigvee_{m \in T'} f(m)$,
  \item $\bigvee_{n \in \NN} f(n) = \top$, and
  \item for all $T, T', T'' \in \Sigma^\NN$, if
    \begin{equation*}
      \textstyle
      \big(\bigvee_{k \in T} b_k\big)
      \land
      \big(\bigvee_{m \in T'} b_m\big) =
      \bigvee_{n \in T''} b_n
    \end{equation*}
    then
    \begin{equation*}
      \textstyle
      \big(\bigvee_{k \in T} f(k)\big)
      \land
      \big(\bigvee_{m \in T'} f(m)\big) =
      \bigvee_{n \in T''} f(n).
    \end{equation*}
  \end{enumerate}
  Conversely, any map $f : \NN \to M$ satisfying these conditions is so induced by a unique $\sigma$-homomorphism $\phi : L \to \Sigma$, characterized by
  \begin{equation*}
    \textstyle
    \phi\big(\bigvee_{n \in T} b_n\big) = \bigvee_{n \in T} f(n).
  \end{equation*}
\end{lemma}

\begin{proof}
  It is easy to check that $f$ really satisfies the stated conditions. Conversely, suppose~$f$ satisfies the conditions. Then the map $\phi$ is well defined by the first condition and the fact that~$L$ is generated by the base. The first condition also guarantees that $\phi$ commutes with countable joins, and the remaining two that it commutes with finite meets.
\end{proof}

\begin{corollary}
  \label{cor:sober-image-equalizer}%
  For a countably based sober space $(X, \topol{B})$, the basic neighborhood filter is the equalizer of two maps whose common codomain is a power of~$\Sigma$.
\end{corollary}

\begin{proof}
  The corollary claims that $\nbhmap[X]$ is an equalizer of the form
  \begin{equation*}
    \xymatrix@+2em{
      {X}
      \ar[r]^{\nbhmap[X]}
      &
      {\Sigma^\NN}
      \ar@<0.25em>[r]
      \ar@<-0.25em>[r]
      &
      \Sigma^{I}
    }
  \end{equation*}
  In plain terms, this amounts to $\nbhmap[X]$ being injective and its image being the solution set of a system of equations between semidecidable truth values. Injectivity holds because $B$ generates $\topol{B}$ and by sobriety points are uniquely determined by their neighborhood filters. The system of equations is essentially the one from  \cref{lemma:when-induced-by-homomorphism} adapted to capture~$\nbhmap[X]$. Specifically, define the maps $J : \Sigma^\NN \to \mathcal{T}$ and $K : \Sigma^\NN \times \Sigma^\NN \to \Sigma$ by
  \begin{equation*}
    \textstyle
    J(T) = \bigcup_{i \in T} B_i
    \quad\text{and}\quad
    K(S, T) = (\some{n \in \NN} n \in S \land n \in T).
  \end{equation*}
  We claim that $S \in \Sigma^\NN$ is in the image of $\nbhmap[X]$ if, and only if,
  \begin{enumerate}
  \item 
    \label{item:sober-regular-mono-join}%
    for all $T, T' \in \Sigma^\NN$ if $J(T) = J(T')$ then $K(S, T) = K(S, T')$,
  \item 
    \label{item:sober-regular-mono-top}%
    $K(S, \NN) = \top$, and
  \item 
    \label{item:sober-regular-mono-meet}%
    for all $T, T', T'' \in \Sigma^\NN$ if
      $J(T) \cap J(T') = J(T'')$ then
      $K(S, T) \land K(S, T') = K(S, T'')$.
  \end{enumerate}
  It is evident that we could rewrite these conditions as a rather unwieldy equality between two semidecidable sets. The first condition says that $K(S, {-})$ factors through~$J$, while the remaining two ensure that the factorization commutes with finite meets. For $S = \nbh{x}$ both conditions are easily checked using the fact that $K(\nbh{x}, T) = (x \in \bigcup_{i \in T} B_i)$.
  Conversely, if the conditions hold for some $S \in \Sigma^\NN$ then we may define $p : \topol{B} \to \Sigma$ by
  \begin{equation*}
    \textstyle
    p \left(\bigcup_{i \in T} B_i \right) = K(S, T),
  \end{equation*}
  which is a $\sigma$-homomorphism by \cref{lemma:when-induced-by-homomorphism}. As $X$ is sober there exists $x \in X$ such that $p = \hat{x}$, and so, for every $n \in \NN$,
  \begin{equation*}
    (n \in S) =
    K (S, \set{n}) =
    p(B_n) =
    (x \in B_n),
  \end{equation*}
  which proves $S = \nbh{x}$.
\end{proof}

\subsection{The synthetic KLST theorem}
\label{sec:pointwise-continuity-theorem}

The classic KLST theorem states that all computable maps from a computable complete separable metric space to a computable metric space are computably continuous. Dieter Spreen generalized it in several ways to computable maps between certain effective $T_0$-spaces. We aim for the version of the theorem that speaks about maps into regular spaces, so we first need a constructive version of regularity.

A topological space $(X, \mathcal{T})$ is \defemph{(pointwise) regular} when for every $x \in U \in \mathcal{T}$ there are disjoint open sets $S, T \in \mathcal{T}$ such that $x \in S \subseteq U$ and $T \cup U = X$.

\begin{proposition}
  \label{prop:separable-metric-regular}
  A separable metric space is pointwise regular.
\end{proposition}

\begin{proof}
  Let $(M, d)$ be a metric space with a dense sequence $(x_n)_{n \in \NN}$. It suffices to verify the regularity condition for an open ball, so suppose $x \in B(y, r)$. Let $q = d(x,y)$, $S = B(x, (r - q)/3)$ and $T = \set{z \in M \such d(y, z) > (2 r + q) / 3}$. Clearly, $S$ and $T$ are disjoint, $x \in S$ and $B(y,r) \cup T = M$. The set $T$ is open because it is the union of the countable family of basic open balls
  \begin{equation*}
    \set{B(x_n, 2^{-k}) \such
      n, k \in \NN \land d(x_n, y) > 2^{-k} + (2 r + q) / 3
    }. \qedhere
  \end{equation*}
\end{proof}

If pointwise regularity is to be imposed on the codomain, what condition should restrict the domain of the map appearing in the KLST theorem? By analyzing Spreen's proof and his notion of ``witness for non-inclusion'', we obtain the following notion.

\begin{definition}
  \label{def:spreen-space}%
  A \defemph{Spreen space} is a space $(X, \mathcal{T})$ in which every point separated from an overt subset by a semidecidable subset is also separated from it by an open one.
\end{definition}

\noindent
Precisely: if $T \subseteq X$ is overt $x \in S \in \Sigma^X$ and $S \cap T = \emptyset$, then there is $U \in \mathcal{T}$ such that $x \in U$ and $U \cap T = \emptyset$.

Thus, a Spreen space is one in which the pointwise open sets, while coarser than the semidecidable ones, are still fine enough to witness non-inclusion in an overt subset (hence the original name). At the moment we are only able to provide trivial examples, such as a space equipped with the intrinsic topology. Nevertheless, the KLST theorem is expressed naturally using Spreen spaces.

\begin{theorem}[KLST]
  \label{thm:synthetic-klst}%
  A map from an overt Spreen space to a regular space is pointwise continuous.
\end{theorem}

\begin{proof}
  Consider any $f : X \to Y$, as in the statement of the theorem. Let $x \in X$ and $f(x) \in V \subseteq Y$ where $V$ is open. Because $Y$ is regular, there are disjoint open sets $S, T \subseteq Y$ that $f(x) \in S \subseteq V$ and $T \cup V = Y$. Notice that the inverse image $\inv{f}(T)$ is a semidecidable subset of an overt space, hence overt. Then $x \in \inv{f}(S)$, and $\inv{f}(S)$ is semidecidable and disjoint from the overt $\inv{f}(T)$. Therefore, there is an open $U \subseteq X$ such that $x \in U$ and $U \cap \inv{f}(T) = \emptyset$, but then $\inv{f}(T) \cup \inv{f}(V) = X$ implies $U \subseteq \inv{f}(V)$, as required.
\end{proof}

It is clear from the proof that we could formulate variations of the theorem. For example, we could require overtness of the codomain instead of the the domain. Or we could drop overtness altogether and replace regularity of the codomain~$Y$ with the following condition: if $y \in V$ and $V \subseteq Y$ is open, then there exists disjoint semidecidable~$S$ and overt~$T$ such that $y \in S$ and $V \cup T = Y$. Is this a relevant notion of space?

In order to recover the classic KLST theorem we would have to know that a complete separable metric space is an overt Spreen space. This is a properly intuitionistic requirement, because it contradicts classical logic. Therefore, we work toward showing that in the context of synthetic computability theory, or the effective topos, there is a rich supply of Spreen spaces.

\section{Synthetic computability}
\label{sec:synth-comp}

Everything we have done so far is valid in any topos satisfying Dependent Choice, as well as in Bishop-style constructive mathematics. We now move to the effective topos. However, rather than working directly with the topos, we identify axioms which are valid in it, and then keep working constructively using the additional axioms. We call this setting \emph{synthetic computability}~\cite{Bauer06,bauer17}. It is a reformulation and extension of Fred Richman's setup from~\cite{richman83}.

Double negation plays an important role in synthetic computability. Say that a proposition $\phi$ is \defemph{$\lnot\lnot$-stable}, or just \defemph{stable}, if $\lnot\lnot \phi \lthen \phi$. Similarly, a subset $S \subseteq X$ is stable when its membership predicate is stable, $\lnot\lnot(x \in S) \lthen x \in S$ for all $x \in X$.

The first axiom is the familiar
\begin{quote}
  \emph{\textbf{Markov principle:} If not all terms of a binary sequence are zero, then
    the sequence contains a one.}
\end{quote}
The axiom is equivalent to the statement that semidecidable truth values are $\neg\neg$-stable, i.e., $\all{p \in \Sigma} \neg\neg p \lthen p$.


The second axiom corresponds to the fact that there is an computable enumeration of c.e.~sets:
\begin{quote}
  \emph{\textbf{Enumerability Axiom:} There are countably many countable subsets of
    natural numbers.}
\end{quote}
As is traditional in computability theory, we let~$W$ denote a fixed enumeration of $\Sigma^\NN$. The axiom is equivalent to Fred Richman's axiom CFP which states that there are countably many partial maps with countable graphs~\cite{richman83}.

The last axiom needs an introduction. A Turing machine may be run on any input tape, even one that it is not designed for. It may get stuck, diverge or output nonsense, but it will always do something. This observation is so trivial that it is rarely made explicit, yet it is essential in many computability theory proofs and constructions. Its topological manifestation is the following notion.

\begin{definition}
  A subset $X \subseteq Y$ is an \defemph{intrinsic subspace} of~$Y$ when
  the restriction map ${-} \cap X : \Sigma^Y \to \Sigma^X$ is surjective.
\end{definition}

\noindent
In other words, every $S \in \Sigma^X$ is the restriction of some $T \in \Sigma^Y$.
Our last axiom states:
\begin{quote}
  \emph{\textbf{Stable Subspace Axiom:}}
    A stable subset of the natural numbers is an intrinsic subspace.
\end{quote}
In the effective topos a $\lnot\lnot$-stable subobject of~$\NN$ is just an ordinary subset $X \subseteq \NN$, where each $n \in X$ is realized by itself. The object $\Sigma^X$ is the numbered set whose elements are restrictions of c.e.~sets to~$X$, and a number~$k$ realizes $U \in \Sigma^X$ when $U = \set{n \in X \such \cf_k(n) \defined}$. Because it makes sense to apply the $k$-th partial computable map $\cf_k$ to any number, not just those in~$X$, $k$ also realizes a c.e.~set~$V \subseteq \NN$ which restricts to~$U$, hence the stable subspace axiom is realized by the identity function.

Henceforth we adopt the above axioms, unless otherwise stated. Anyhow, we shall make all their applications explicit.

\subsection{Partial and multi-valued maps}
\label{sec:partial-multi-valued}

The Stable Subspace Axiom allows us to extend semidecidable sets, but sometimes we wish to extend maps, and partial maps in particular. We also define multivalued maps, which we will need later on.

A \defemph{partial map} $f : X \parto Y$ is a map $f : \supp{f} \to Y$ defined on a subset $\supp{f} \subseteq X$, called the \defemph{support} of~$f$. Equivalently, it is a map $f : X \to \tilde{Y}$ whose values are subsets of~$Y$ with at most one element,
\begin{equation*}
  \tilde{Y} = \set{S \subseteq Y \such \all{y,y' \in S} y = y'}.
\end{equation*}
The elements of $\tilde{Y}$ are called \defemph{partial values}. The empty set $\emptyset \in \tilde{Y}$ plays the role of the undefined value, and the singletons $\set{y}$ with $y \in Y$ the total values.
The support of a partial map $f : X \to \tilde{Y}$ is computed as $\supp{f} = \set{x \in X \such \some{y \in Y} y \in f(x)}$.

A \defemph{multivalued} map $f : X \multito Y$ is a map $f : X \to \inhab{Y}$ whose values are inhabited subsets
of~$Y$,
\begin{equation*}
  \inhab{Y} = \set{S \subseteq Y \such \some{y \in Y} y \in S}.
\end{equation*}

The partial maps $\NN \to \tilde{\NN}$ in the effective topos do \emph{not} correspond to the
partial computable maps. To get the desired correspondence we must take partial maps whose support is semidecidable. 
For this purpose we define the \defemph{(semidecidable) lifting $\lift{Y}$} of~$Y$ to be the set of partial values whose inhabitation is semidecidable,
\begin{equation*}
  \lift{Y} = \set{S \in \tilde{Y} \such (\some{y \in Y} y \in S) \in \Sigma}.
\end{equation*}
The support of a map $f : \NN \to \lift{\NN}$ is indeed semidecidable because $n \in \supp{f}$ is equivalent to $\some{k \in \NN} k \in f(n)$, which is semidecidable by definition.
We shall only consider partial maps $X \to \lift{Y}$.






The Stable Subspace Axiom allows us to extend partial maps.

\begin{proposition}
  \label{prop:extend-f}%
  A map $f : T \to \lift{\NN}$ whose domain $T \subseteq \NN$ is a stable subset has an
  extension $\bar{f} : \pcm$.
\end{proposition}

\begin{proof}
  For every $n \in \NN$, $\inv{f}(\set{n}) = \set{m \in T \such n \in f(m)}$ is semidecidable, and so it has a semidecidable extension to $\NN$ by the Stable Subspace Axiom. Use countable choice twice, first to obtain a map $S : \NN \to \Sigma^\NN$ such that $S(n) \cap T = \inv{f}(\set{n})$ for every $n \in \NN$, and then a map $e : \NN \to \one + \NN$ such that $k \mapsto e(\pair{n,k})$ enumerates $S(n)$ for every $n \in \NN$.
  The desired extension $\bar{f} : \pcm$ may be defined by
  \begin{equation*}
    \bar{f}(m) = \set{n \in \NN \such
    \some{i \in \NN}
      e(i) = m \land \pi_1(i) = n \land
      \all{j < i} e(j) \neq m
    }. \qedhere
  \end{equation*}
\end{proof}

The use of Countable Choice in the previous proposition is essential for getting a single-valued extension. Without it we can only hope to get a multivalued one, as demonstrated by the following proposition.

\begin{proposition}
  \label{prop:multi-extension}
  Suppose $X \subseteq Y$ is a stable subset of a countable set $Y$. Then every $U : X \to \Sigma$ extends to a multivalued $\overline{U} : Y \multito \Sigma$.
\end{proposition}

\begin{proof}
  Let $e : \NN \to 1 + Y$ be an enumeration of~$Y$ and let $S = \inv{e}(X) = \set{n \in \NN \such e(n) \in X}$. Because $S$ is a stable subset of $\NN$ and $V = \set{n \in S \such e(n) \in U}$ is semidecidable, by the Stable Subspace Axiom~$V$ has an extension $\overline{V} \in \Sigma^\NN$. We may take $\overline{U}(x) = \set{p \in \Sigma \such \some{n \in \NN} e(n) = x \land p = (n \in \overline{V})}$.
\end{proof}

The situation in which semidecidable subsets of a subset extend to multivalued semidecidable subsets of the larger set will figure later on, so we give it a name.

\begin{definition}
  A \defemph{weak intrinsic subspace} $X \subseteq Y$ is a subset such that every $S : X \to
  \Sigma$ has a multivalued extension $\bar{S} : Y \multito \Sigma$. Similarly, an
  injective map $i : X \to Y$ is a \defemph{weak intrinsic inclusion} when every $S : X \to
  \Sigma$ extends along~$i$ to a multivalued $\bar{S} : Y \multito \Sigma$.
\end{definition}

In synthetic computability there are weak intrinsic subspaces which are not intrinsic subspaces.

\begin{proposition}
  \label{prop:weak-counter-example}%
  The inclusion $i : \NNx \to \Sigma^{\NN \times \NN}$ of the Baire space by the map
  $i(\alpha) = \set{(n, \alpha(n)) \such n \in \NN}$ is a weak intrinsic inclusion which is not an intrinsic inclusion.
\end{proposition}

\begin{proof}
  Clearly, $i$~is an injection, and an application of Markov principle shows that its image is stable. Thus, $\NNx$ is a weak intrinsic subset of $\Sigma^{\NN \times \NN}$ by \cref{prop:multi-extension}. But it is not an intrinsic subspace. Indeed, by \cref{prop:sigma-N-scott-topology}, proved below, every semidecidable subset of $\Sigma^\NN$ is a union of basic opens whose inverse images under~$i$ are open balls in $\NNx$, for the ultrametric~\eqref{eq:cantor-metric}.
  However, by~\cite[Corollary 7.3]{BauerLesnik12} there are semidecidable subsets of~$\NNx$ which are not unions of balls.
\end{proof}

\subsection{Recursion theorem and its consequences}
\label{sec:recurs-theor-its}

The synthetic recursion theorem is reminiscent of Lawvere's fixed point theorem~\cite{lawvere69} which states that any endomap on $X$ has a fixed point if there is a surjection $A \to X^A$.
Say that $X$ has the \defemph{multivalued fixed-point property} when for every $f : X \multito X$ there is $x \in X$ such that $x \in f(x)$, called a \defemph{fixed point} of~$f$.

\begin{theorem}[Recursion theorem]
  \label{thm:recursion-theorem}%
  If there is a surjection $\NN \to X^\NN$ then $X$ has the multivalued fixed-point property.
\end{theorem}

\begin{proof}
  Let $e : \NN \to X^\NN$ be an enumeration and $f : X \multito X$ a multivalued map whose fixed point needs to be constructed. For every $n \in \NN$ there is $x \in X$ such that $x \in f(e(n)(n))$, therefore by countable choice there is a map $g : \NN \to X$ such that $g(n) \in f(e(n)(n))$ for every $n \in \NN$. There is $k \in \NN$ such that $g = e(k)$. But now $e(k)(k)$ is a fixed point of~$f$ because $e(k)(k) = g(k) \in f(e(k)(k))$.
\end{proof}

The theorem is discussed in detail in~\cite{bauer17}, where a justification for its name can be found, too.
Both Lawvere's fixed-point theorem and the synthetic recursion theorem are valid in pure constructive mathematics, but have no interesting instances without extra axioms. One such axiom is the Enumeration Axiom, because it is precisely the hypothesis of the recursion theorem for~$\Sigma$. Also, since $\Sigma^\NN \cong \Sigma^{\NN \times \NN} \cong (\Sigma^\NN)^\NN$, the space $\Sigma^\NN$ has the multivalued fixed-point property, too.

We are interested in topological consequences of the recursion theorem. Let $\NNx$ be the set of antimonotone binary sequences
\begin{equation*}
  \NNx = \set{t \in \two^\NN \such \all{n \in \NN} t(n) \geq t(n+1)}.
\end{equation*}
It helps to think of $\NNx$ as the one-point compactification of~$\NN$. A natural number~$n$ corresponds to the sequence $\overline{n} : \NN \to \two$ defined by $\overline{n}(i) = 1 \liff i < n$ that drops to zero after~$n$ ones. The point at infinity $\infty$ is the sequence of all ones. We do not distinguish notatationally between a number $n$ and the corresponding element $\overline{n}$ of~$\NNx$, and think of $\NN$ as a subset of $\NNx$.

The order on $\NNx$ defined by
\begin{equation*}
  t \leq u \iff \all{n \in \NN} t(n) \leq u(n)
\end{equation*}
extends the usual one on~$\NN$, and has $\infty$ as the largest element.
There is also a strict order
\begin{equation*}
  t < u \iff \some{n \in \NN} t(n) < u(n).
\end{equation*}
For $n \in \NN$ and $t \in \NNx$, it is decidable whether $n < t$, because it is equivalent to $t(n) = 1$.
Likewise, $t \leq n$ is decidable because it is equivalent to $t(n) = 0$. Both facts together imply that, for all $n \in \NN$ and $t \in \NNx$, either $n < t$ or $t \leq n$.


Rosolini's dominance is a quotient of $\NNx$ by the map $t \mapsto (t < \infty)$. Indeed, $t < \infty$ is semidecidable because it means $\some{n \in \NN} t(n) = 0$, while for any $\alpha \in \two^\NN$ we have
\begin{equation*}
  (\some{n \in \NN} \alpha(n) = 1)
  \iff
  r(\alpha) < \infty,
\end{equation*}
where $r : \two^\NN \to \NNx$ is the retraction $r(\alpha)(n) = \min (\alpha(0), \ldots,
\alpha(n))$.



When we think of an element $t \in \NNx$ as the ``time'' at which a semidecidable condition is satisfied, we recognize in the quotient $\NNx \to \Sigma$ the ``waiting arguments'' of classical computability theory, in which one waits for a semidecision procedure to terminate. A related kind of argument is a (possible unsuccessful) search for a value satisfying a semidecidable condition, which we express as follows in the synthetic setting.

\begin{lemma}
  \label{lemma:extraction}
  Suppose $\some{x \in X} \phi(x)$ is semidecidable. Then there is $T \in \lift{X}$ such that
  $\some{x \in X} \phi(x)$ is equivalent to $\some{x \in T} \phi(x)$.
\end{lemma}

\begin{proof}
  There exists $c \in \NN \to \two$ such that $\some{x \in X} \phi(x)$ is equivalent to $\some{n \in \NN} c_n = 1$. For every $n \in \NN$ there is $u \in 1 + X$ such that
  \begin{equation*}
    c_n = 1 \liff (u \in X \land \phi(u)).
  \end{equation*}
  Indeed, if $c_n = 0$ then we take $u = \star$, and if $c_n = 1$ then there is $x \in X$ such that $\phi(x)$, so we take $u = x$. By Countable Choice there is a sequence $b : \NN \to \one + X$ such that $c_n = 1 \liff b_n \in X \land \phi(b_n)$ for all $n \in \NN$.
  Now define
  \begin{equation*}
    T = \set{x \in X \such \some{n \in \NN} b_n = x \land \all{k < n} b_k = \star}.
  \end{equation*}
  Clearly, any two elements of $T$ are equal, and $\some{x \in X}{x \in T}$ is semidecidable because it is equivalent to the statement $\some{n \in \NN} c_n = 1$. Thus we have $T \in \lift{X}$, as required.

  Notice that every $x \in T$ satisfies $\phi(x)$, therefore $\some{x \in T} \phi(x)$ implies $\some{x \in X} \phi(x)$. Conversely, if there is $x \in X$ such that $\phi(x)$ then there is a least~$n \in \NN$ such that $c_n = 1$,
  which means that $b_n \in T$.
\end{proof}

In~\cite{BauerLesnik12} the following continuity principle is shown to hold in many (necessarily non-classical) varieties of constructive mathematics, including synthetic computability.

\begin{theorem}[WSO]
  \label{thm:wso}%
  If $U : \NNx \to \Sigma$ is such that $\infty \in U$ then $n \in U$ for some $n \in \NN$.
\end{theorem}

We need the generalization of the principle to multivalued maps.

\begin{theorem}[Multivalued WSO]
  \label{thm:multi-wso}
  If $U : \NNx \multito \Sigma$ is such that $U(\infty) = \set{\top}$ then there is $n \in
  \NN$ such that $\top \in U(n)$.
\end{theorem}

\begin{proof}
  Given $U$ as in the statement of the theorem, define the multivalued map $f : \Sigma \multito \Sigma$ by
  \begin{equation*}
    f(p) = \set{q \in \Sigma \such \some{t \in \NNx} p = (t < \infty) \land q \in U(t)}.
  \end{equation*}
  By recursion theorem there is a fixed point $p \in f(p)$, and so there is $x \in \NNx$ such that $p = (x < \infty)$ and $p \in U(x)$. It suffices to show that $x < \infty$. If $x = \infty$ then $\bot = p \in U(\infty)$, which contradicts the assumption. Therefore $x \neq \infty$ and by Markov principle $x < \infty$.
\end{proof}

We shall use the following corollary repeatedly.

\begin{corollary}
  \label{cor:wso}
  If $f : \NNx \to X$ and $U : X \multito \Sigma$ is such that $U(f(\infty)) = \set{\top}$ then $\top \in U(f(n))$ for some $n \in \NN$.
\end{corollary}

\begin{proof}
  Apply \cref{thm:multi-wso} to $U \circ f$.
\end{proof}

With WSO in hand we can compute the intrinsic topology of $\Sigma^\NN$.

\begin{proposition}
  \label{prop:sigma-N-scott-topology}
  The intrinsic topology of $\Sigma^\NN$ is the Scott topology.
\end{proposition}

\begin{proof}
  Recall that the Scott topology on $\Sigma^\NN$ is the topology generated by the basic open sets of the form
  \begin{equation*}
    \upper T =
    \set{S \in \Sigma^\NN \such T \subseteq S}
  \end{equation*}
  where $T$ is a finite subset of~$\NN$. So for any $U : \Sigma^\NN \to \Sigma$ we have to prove that for all $S \in \Sigma^\NN$, $S \in U$ precisely when some finite subset of $S$ is already in $U$,
  \begin{equation*}
    S \in U \iff
    \some{T \subseteq S} T \in U \land \text{$T$ finite}.
  \end{equation*}
  For the implication from left to right, consider the map $f : \NNx \to \Sigma^\NN$ defined by
  \begin{equation*}
    f(t) = \set{k \in \NN \such k < t \land n \in S}.
  \end{equation*}
  Since $f(\infty) = S \in U$ there exists by WSO some $k < \infty$ such that $f(k) \in U$, and so $f(k)$ is the finite subset of $S$ we are looking for.

  For the converse, suppose $T \subseteq S$ is a finite subset of $S$ and $T \in U$. Consider the map $f : \NNx \to \Sigma^\NN$ defined by
  \begin{equation*}
    \textstyle
    f(t) = \bigcup_{i \in \NN} (\text{if $i < t$ then $T$ else $S$}).
  \end{equation*}
  As $f(\infty) = T \in U$ there exists by WSO some $k < \infty$ such that $S = f(k)
  \in U$.
\end{proof}

\subsection{Weak intrinsic subspaces of $\Sigma^\NN$}
\label{sec:weak-intrinsic-sigma-N}

Our plan is to transfer topological properties of $\Sigma^\NN$ to a countably based $T_0$-space~$X$ via the
basic neighborhood filter $\nbhmap : X \to \Sigma^\NN$, which was shown to be a topological embedding in \cref{prop:sigma-N-universal}.
However, since Spreen spaces involve semidecidable sets, not just open ones, we need to relate $\nbhmap[X]$ to semidecidable subsets, too. Ideally, $\nbhmap[X]$ would be an intrinsic embedding, but that is to stringent a requirement, as it would eliminate examples that we would like to keep, see \cref{prop:weak-counter-example}.
Luckily, $\nbhmap[X]$ being just a weak intrinsic inclusion suffices for our purposes. Dieter Spreen identified a sufficient condition for this to be the case, which he called effective limit passing.

\begin{definition}
  A countably based space is a \defemph{space with limit passing} when it is a $T_0$-space and the image of its basic neighborhood filter is stable.
\end{definition}

In symbols, $(X, \topol{B})$ has limit passing when, for all $S \in \Sigma^\NN$, if $\lnot\lnot\some{x \in X} S = \nbh[X]{x}$ then there exists a unique $x \in X$ such that $S = \nbh[X]{x}$. 
From a computational point of view, the limit passing condition says that a point $x \in X$ can be recovered when its neighborhood filter $\nbh[X]{x}$ qua semidecidable predicate on the indices of basic open sets. Here are some spaces with limit passing.

\begin{proposition}
  \label{prop:cb-sober-limit-passing}
  \mbox{}
  \begin{enumerate}
  \item A countably based sober space is a space with limit passing.
  \item A stable subspace of a space with limit passing is again such a space.
  \end{enumerate}
\end{proposition}

\begin{proof}
  \mbox{}
  \begin{enumerate}
  \item 
    Clearly, a countably based sober space $(X, \topol{B})$ is a $T_0$-space. By \cref{cor:sober-image-equalizer} there is an equalizer
    \begin{equation*}
      \xymatrix{
        {X}
        \ar[r]^{\nbhmap[X]}
        &
        {\Sigma^\NN}
        \ar@<0.25em>[r]^{f}
        \ar@<-0.25em>[r]_{g}
        &
        \Sigma^{I}
      }
    \end{equation*}
    This implies that the image of $\nbhmap[X]$ is a set of the form
    \begin{equation*}
      \set{S \in \Sigma^\NN \such
        \all{i \in I} f(S)(i) = g(S)(i)
      }.
    \end{equation*}
    By Markov principle equality on $\Sigma$ is stable, therefore the set is defined by a stable predicate on $\Sigma^\NN$.
    
  \item Let $(Y, \topol{B})$ be a space with limit passing and $X \subseteq Y$ a stable subset. The subspace topology on $X$ is induced by the base $B'_n = X \cap B_n$. Thus, for $x \in X$ we have $\nbh[X]{x} = \nbh[Y]{x}$ and so the image of $\nbhmap[X]$ is a stable subset of the image of $\nbhmap[Y]$, which is itself a stable subset.
  Of course $(X, \topol{B'})$ is a $T_0$-space because $T_0$ is a hereditary property. \qedhere
  \end{enumerate}
\end{proof}

Our first application of limit passing relates separability and overtness.

\begin{proposition}
  \label{prop:sober-separable-ershov}%
  If a sequence is dense with respect to a sober topology, then it is dense with respect
  to intrinsic topology.
\end{proposition}

\begin{proof}
  Let $(X, \topol{B})$ be a sober space with a dense sequence $(x_n)_{n \in
    \NN}$ and $y \in S \in \Sigma^X$. Combine \cref{prop:cb-sober-limit-passing} and \cref{prop:multi-extension}
  to extend~$S$ to $\bar{S} : \Sigma^\NN \multito \NN$ along $\nbhmap$.
  Let $a : \NN \to \NN$ be an enumeration of $y$.
  By Countable Choice there is a sequence $(i_n)_{n \in \NN}$ such that, for all $n \in \NN$,
  \begin{equation*}
    x_{i_n} \in B_{a(0)} \cap \cdots \cap B_{a(n)}.
  \end{equation*}
  Define $f : \NNx \to \Sigma^\NN$ by
  \begin{equation*}
    f(t) = \set{a(i) \such i < t} \cup \set{k \in \NN \such t < \infty \land k \in \nbh{x_{i_t}}}.
  \end{equation*}
  Note that $f(\infty) = \nbh{y}$ and $f(n) = \nbh{x_{i_n}}$ for $n \in \NN$.
  Because $\bar{S}(f(\infty)) = \set{\top}$ there is $n \in \NN$ such that
  $\top \in \bar{S}(f(n)) = \set{S(x_{i_n})}$, which gives the desired $x_{i_n} \in S$.
\end{proof}

\begin{corollary}
  \label{prop:separable-overt}
  Separable sober spaces, such as separable complete metric spaces and $\omega$-algebraic {\wCPOs}, are overt.
\end{corollary}

\begin{proof}
  If $(X, \mathcal{T})$ is sober and $(x_n)_{n \in \NN}$ a sequence in~$X$ which is dense with respect to~$\mathcal{T}$, then by \cref{prop:sober-separable-ershov} it is also dense for the intrinsic topology~$\Sigma^X$. But then for any $S \in \Sigma^X$ we have
  \begin{equation*}
    (\some{y \in X} y \in S) \iff
    (\some{n \in \NN} x_n \in S),
  \end{equation*}
  and the right-hand statement is semidecidable.
\end{proof}

\section{Spreen spaces in synthetic computability}
\label{sec:spreen-spaces-in-synthetic-computability}

Finally, here is a supply of Spreen spaces.

\begin{theorem}
  \label{thm:cb-sober-spreen}%
  A countably based sober space is a Spreen space.
\end{theorem}

\begin{proof}
  Let $(X, \topol{B})$ be sober, $S \subseteq X$ semidecidable, $T \subseteq X$ overt, $S \cap T = \emptyset$, and $x \in S$. We seek a basic neighborhood $B_m$ of $x$ which is disjoint from $T$.

  Let $a : \NN \to \NN$ be an enumeration of $\nbh{x}$, and define
  $C_n = T \cap B_{a(0)} \cap \cdots \cap B_{a(n)}$.
  \Cref{lemma:extraction} and Countable choice together yield a map $g : \NN \to \lift{X}$ such that $g(n) \subseteq C_n$ for every $n \in \NN$, and moreover $g(n)$ is inhabited if, and only if, $C_n$ is inhabited.

  Define $f : \NNx \to \Sigma^\NN$ by
  \begin{equation*}
    f(t) =
    \set{a(i) \such i < t} \cup
    \set{k \in \NN \such t < \infty \land \some{y \in X} y \in g(t) \cap B_k}.
  \end{equation*}
  Notice that $f(\infty) = \nbh{x}$, while for every $n \in \NN$ we have:
  if $C_n$ is inhabited then $f(n) = \nbh{y}$ for some $y \in C_n$.

  Using \cref{prop:cb-sober-limit-passing} and \cref{prop:multi-extension},
  extend~$S$ along~$\nbhmap[X]$ to $\bar{S} : \Sigma^\NN \multito \Sigma$.
  Because $\bar{S}(f(\infty)) = \set{S(x)} = \set{\top}$ there is $n \in \NN$ such that $\top \in \bar{S}(f(n))$.
  If $C_n$ were inhabited then for some $y \in C_n$ we would have $f(n) = \nbh{y}$, hence $\top \in \bar{S}(f(n)) = \set{S(y)}$, from which the contradiction $y \in S \cap T = \emptyset$ would follow.
  Therefore, $C_n = \emptyset$, so it suffices to take any basic open $B_m$ such that $x \in B_m \in B_{a(0)} \cap \cdots \cap B_{a(n)}$.
\end{proof}

In conclusion, here is the classic KLST theorem.

\begin{corollary}[Classic KLST]
  Every map from a complete separable metric space to a metric space is pointwise continuous.
\end{corollary}

\begin{proof}
  A complete separable metric space is sober by \cref{prop:csm-sober}, overt by \cref{prop:separable-overt} and a Spreen space by \cref{thm:cb-sober-spreen}. Because a metric space is pointwise regular by \cref{prop:separable-metric-regular}, we may apply \cref{thm:synthetic-klst}, the synthetic KLST theorem.
\end{proof}

\subsubsection*{Acknowledgment}

I thank Dieter Spreen for explaining effective spaces to me when we first met in 1997 at Mathematical Foundations of Programming Semantics in Pittsburgh, PA. Dieter has always been very kind to me, and helped me with entering the research community of which I am a proud member today.

This material is based upon work supported by the Air Force Office of Scientific Research under award number FA9550-21-1-0024.


\bibliographystyle{amsplain}
\bibliography{kls}

\end{document}